\documentclass{article}

\usepackage{amssymb,amsmath,latexsym,amsthm}

\newcommand{\Q}{\mathbb{Q}}
\newcommand{\QQ}{\mathcal{Q}}
\newcommand{\Z}{\mathbb{Z}}
\newcommand{\RR}{\mathcal{R}}
\newcommand{\N}{\mathbb{N}}
\newcommand{\T}{\mathbb{T}}
\newcommand{\G}{G_\Q}
\newcommand{\OO}{\mathcal{O}}
\newcommand{\im}{\mathrm{Im}}
\newcommand{\new}{\mathrm{new}}
\newcommand{\Frob}{{\rm Frob }}
\newcommand{\trace}{{\rm trace}}

\newtheorem{lemma}{Lemma}
\newtheorem{corollary}{Corollary}
\newtheorem{proposition}{Proposition}
\newtheorem{definition}{Definition}
\newtheorem{theorem}{Theorem}

\begin{document}
\def\refname{\centerline{Bibliography}}

\title{Congruences between modular forms and lowering the level mod $\ell^n$}

\author{
Luis Dieulefait\\
{\small{
Dept. d'\`{A}lgebra i Geometria, Universitat de Barcelona}}\\
{\small{Gran Via de les Corts Catalanes 585 -- 08007 Barcelona, Catalonia, Spain}}\\
{\small{e-mail: ldieulefait@ub.edu}}\\
{\small{\texttt{http://atlas.mat.ub.es/personals/dieulefait}}}
\bigskip\\
Xavier Taix\'{e}s i Ventosa\\
{\small{Institut f\"ur Experimentelle Mathematik, Universit\"at
Duisburg-Essen}}\\
{\small{Ellernstra{\ss}e 29 -- 45326 Essen, Germany}}\\
{\small{e-mail: xavier@iem.uni-due.de}}\\
{\small{\texttt{http://www.exp-math.uni-essen.de/\textasciitilde
xavier}}}}

\maketitle

\begin{abstract}
In this article we study the behavior of inertia groups for
modular Galois mod $\ell^n$ representations and in some cases we give a
generalization of Ribet's lowering the level result (cf. \cite{Ribet1990}).
\end{abstract}

\bigskip
\section{Introduction}
Let $f=q+\sum_{2}^\infty a_iq^i$ and $g=q+\sum_{2}^\infty b_iq^i$ be
two newforms of weight $2$, trivial Nebentypus character and level
$N_f$ and $N_g$ respectively. Let $K_f$ and $K_g$ be the fields
generated by the coefficients of $f$ and $g$, and let $K$ be their
composite field. We denote by $\OO_f$, $\OO_g$ and $\OO$ their rings
of integers. Let $\ell>2$ be a prime and let $\rho_f$ (resp.
$\rho_g$) be the $2$-dimensional $\ell$-adic representation
associated to $f$ (resp. $g$), with values in
$\OO_{f,\ell}:=\OO_f\otimes\Z_\ell$ (resp. $\OO_{g,\ell}$).

Recall that the representation $\rho_f$ ramifies exactly at the
primes in the level of $N_f$ and at $\ell$. For any unramified prime
$t$, the  image of the arithmetic Frobenius $\Frob \; t$ has $\trace
(\rho_f (\Frob \; t)) = a_t$, the Fourier coefficient and $t$-th
Hecke eigenvalue of $f$. Also, the determinant of $\rho_f$ is the
$\ell$-adic cyclotomic character $\chi$.

For a given integer $n$, we use the projection
\[
\OO_{f,\ell}\rightarrow \OO_{f,\ell}/\ell^n\OO_{f,\ell}
\]
and we semi-simplify to obtain the mod $\ell^n$ representation
\[
\overline\rho_{f,\ell^n}:\G\rightarrow
GL_2(\OO_{f,\ell}/\ell^n\OO_{f,\ell}).
\]

Using the decomposition of $\ell$ in $K_f$,
$\ell=\lambda_1^{e_1}\cdot\ldots\cdot\lambda_k^{e_k}$ and the
projection
\[
\prod\OO_{f,\lambda_i}/\lambda_i^{e_i
n}\OO_{f,\lambda_i}\rightarrow\OO_{f,\lambda_i}/\lambda_i^{n}\OO_{f,\lambda_i}
\]
we obtain the mod $\lambda^n$ representation attached $f$ for a
fixed place $\lambda\mid\ell$ in $K_f$
\[
\overline\rho_{f,\lambda^n}:\G\rightarrow
GL_2(\OO_{f,\lambda}/\lambda^n\OO_{f,\lambda}).
\]

Let us fix a place $\lambda\mid\ell$ in $K$ and let us denote also
by $\lambda$ its restrictions to $K_f$ and to $K_g$.

From now on let us assume that the mod $\lambda$ representation
$\overline\rho_{f,\lambda}$  is irreducible (then
$\overline\rho_{f,\lambda}$ is odd and absolutely irreducible), that
$N_g \mid N_f$ and that $\ell \nmid N_f$.

If we take the ideal $\lambda^n\subset\OO$ and the projection
\[
\pi:\OO\rightarrow\OO/\lambda^n,
\]
then we say that two numbers $\alpha\in\OO_f$ and $\beta\in\OO_g$
are congruent modulo $\lambda^n$ if $\pi(\alpha)=\pi(\beta)$.

\begin{definition}
$f$ and $g$ are \textbf{congruent} modulo $\lambda^n$ if $a_p\equiv
b_p \mod\lambda^n$ for almost every prime $p$.
\end{definition}

In fact, this is equivalent to say that their associated mod
$\lambda^n$ Galois representations are isomorphic.

\begin{theorem}
$f\equiv g\mod\lambda^n \Longleftrightarrow
\overline\rho_{f,\lambda^n}\sim\overline\rho_{g,\lambda^n}$.
\end{theorem}

This is just an automatic consequence of Cebotarev's density theorem
since we are assuming that the traces of the images of almost all
Frobenius elements are congruent to each other, and the
Brauer-Nesbitt theorem guarantees that these elements determine the
representation for $\ell>2$. Observe that we do not have to consider
the semi-simplifications of the mod $\lambda^n$ representations
since we are assuming that they are irreducible.

Given a representation $\rho$, let $n_{\rho,p}$ be the conductor of
$\rho$ in the prime $p$. In \cite{Carayol1989}, Carayol studies for
a given mod $\ell$ representation, how much the conductor of a
deformation can increase. He proves the following result.

\begin{proposition}
Let $N=p_1^{n_{p_1}}\ldots p_k^{n_{p_k}}$ and
$\overline{N}=p_1^{\overline{n}_{p_1}}\ldots
p_k^{\overline{n}_{p_k}}$ be the conductors of a $\lambda$-adic
representation $\rho$ and the corresponding  mod $\lambda$
representation $\overline{\rho}_\lambda$, respectively.  Let $p$ be
a prime dividing $N$, $p \neq \ell$, and suppose $\rho$ is such that
$n_p>\overline{n}_p$. Then locally at $p$ $\rho$ is of one of the
following types
\begin{enumerate}
\item $\rho_p=\mu\oplus v$, with $n_{\mu,p}=1$ and
$n_{\overline\mu,p}=0$, and then $n_p=n_{v,p}+1$
\item $\rho_p=\mu\otimes sp(2)$, with $n_{\mu,p}=0$, and then $n_p=1$.
\item $\rho_p=\mu\otimes sp(2)$, with $n_{\mu,p}=1$ and
$n_{\overline\mu,p}=0$, and then $n_p=2$.
\item The  irreducible case in which $n_p=2$.
\end{enumerate}
\end{proposition}

In our case, since we are working without nebentypus, the first case
reduces to $\rho_p=\mu\oplus\mu^{-1}$ and then
$n_p=n_{v,p}+1=n_{\mu,p}+1=2$. Since in all the cases $n_p\leq 2$ we
get the following Corollary.

\begin{corollary}
If $f$ and $g$ are congruent mod $\lambda$ with $N_g \mid N_f$, then
for any prime $p$ dividing $N_f$ but not dividing $\ell N_g$, $p^3
\nmid N_f$.
\end{corollary}

More specifically, if we fix a mod $\lambda$ representation
$\overline\rho$ of conductor $\overline N$, the level of all the
modular deformations of $\overline\rho$ with trivial character,
unramified outside $p\overline N$ for a prime $p\nmid\ell\overline
N$ and minimal at $\overline N$, divides $N=p^2\overline N$.

In the following section we state the main results of this article.
They describe, under certain conditions, how the inertia group of
the Galois representations discussed above behave. In the next
section we introduce Taylor-Wiles' Theorem as we need it for our
proof, which will be given in $\S\ref{Seccio proof}$. Finally we
discuss possible developments of this work.

\bigskip \textbf{Acknowledgments}: the second named author wants to
thank  G. B\"{o}ckle for interesting conversations and also Prof. G.
Frey and G. Wiese for their several helpful corrections and remarks.

\section{Main results}
When we have two newforms $f$ and $g$ as in the previous section,
such that they are congruent mod $\lambda$, $N_g \mid N_f$, and $g$
is minimal in the sense that the conductor $\overline N$ of the
residual representations
$\overline\rho_{f,\lambda}\sim\overline\rho_{g,\lambda}$  equals
$N_g$, we ask ourselves the following two related questions: Which
is the biggest $n$ such that $f$ and $g$ are congruent modulo
$\lambda^n$? Once this value of $n$ is known, is there a reason that
explains why $f$ and $g$ are not congruent anymore mod
$\lambda^{n+1}$?

In \cite{Taixes2008} we give an algorithm that answers the first
question for every possible $\lambda$. Theorem \ref{teorema} below
is a result that answers the second question in some cases.

\begin{definition}
Let $L=\Q(\sqrt{(-1)^{(\ell-1)/2}\ell})$. Then
$\overline\rho_{g,\lambda}$ is strongly irreducible if
$\overline\rho_{g,\lambda}|_{G_L}$ is irreducible.
\end{definition}

\begin{proposition}\label{strong irreducible}
Let $\ell>3$. Then $\overline\rho_{g,\lambda}$ is strongly
irreducible.
\end{proposition}
\begin{proof}
Assuming that $\overline\rho_{g,\lambda}$ is irreducible as in our
case, if $g$ is a newform of weight $2$ and $\ell$ does not divide
its level, clearly the residual representation
$\overline\rho_{g,\lambda}$ has Serre's weight $2$. Thus, this gives
a precise information of the action of inertia at $\ell$, and this
is enough to show that $\overline\rho_{g,\lambda}|_{G_L}$ is
irreducible if $\ell >3$. This is proved in \cite{Ribet} as part
of the proof that the dihedral case can not occur for semistable
weight $2$ representations.
\end{proof}

Let us remark that the condition of $\overline\rho|_{G_L}$ being
irreducible for $\ell = 3$ is easily checked just by finding a prime
$p\equiv 2\pmod3$ such that $b_p \not\equiv 0\pmod\lambda$,
$\lambda\mid3$, or equivalently, such that
$\mathrm{Norm}(b_p)\not\equiv0\pmod{3}$ (where $b_p$ is the
$p$-coefficient of $g$).

\begin{theorem}\label{teorema}
Let $\ell,p\nmid N_g,\ \ell>2$ be two different prime numbers. Let
$f$ be in $S_2(p^kN_g)$, $k\ge1$, and let $g\in S_2(N_g)$ be minimal
with respect to $\lambda$ in the sense defined above. Both cusp
forms are assumed to have trivial nebentypus. Suppose that
$\overline\rho_{f,\lambda}\sim\overline\rho_{g,\lambda}$ and they
are irreducible, and assume that for any other $h\in S_2(N_g),\
\overline\rho_{g,\lambda}\not\sim\overline\rho_{h,\lambda}$. If
$\ell=3$, let $L=\Q(\sqrt{-3})$ and suppose that
$\overline\rho_{g,\lambda}|_{G_L}$ is irreducible. Then,
\[
m:=\min\{n\in
\N:\overline\rho_{f,\lambda^n}\not\sim\overline\rho_{g,\lambda^n}\}
 = \min\{n\in \N:\overline\rho_{f,\lambda^n}|_{I_p}\not\sim\overline\rho_{g,\lambda^n}|_{I_p}\}.
\]
\end{theorem}

Hence, what we show is that in many cases the cause of the break of
the congruence when increasing the power of $\lambda$ is due
precisely to the non-triviality of the action of the inertia group
at a prime in $N_f/N_g$. Let us remark that this is specific to the
situation we are in, namely when $N_g$ is a proper divisor of $N_f$.
If this were not the case and $N_f=N_g$ were congruent modulo some
$\lambda^n$ (in \cite{Taixes2008} we compute dozens of examples), it
is clear that the reason of not being congruent anymore modulo
$\lambda^{n+1}$ can not be related to ramification at any place.

Theorem \ref{teorema} can be reinterpreted  as a generalization to
higher exponents of Ribet's Lowering the Level result
\cite{Ribet1990}.

\begin{corollary}[Lowering the level modulo $\lambda^n$]
Let $f$ be a newform of weight $2$, trivial character and level $p^k
N$ ($p\nmid N$) such that for a given $\lambda\nmid2 p N$ and an
integer $n$, $\overline\rho_{f,\lambda^n}$ does not ramify at $p$.
Let us suppose that there exists exactly one newform $g$ of weight
$2$ and level $N$ congruent to $f$ modulo $\lambda$ (Ribet's
lowering the level provides \textbf{at least} one) satisfying the
strong irreducibility condition. Then, lowering the level can be
generalized modulo $\lambda^n$, i.e., $f$ and $g$ are congruent also
modulo $\lambda^n$.
\end{corollary}

In the previous section we saw that there is no congruence between
two newforms of level $N$ and $p^k N$ if $k>2$. In the case $k=1$,
we can rewrite the Theorem as follows.
\begin{corollary}\label{corollary}
With the same conditions as in Theorem \ref{teorema}, let $k=1$.
Then
\[
\rho_f|_{I_p}=<\left(
              \begin{array}{cc}
                1 & a \\
                0 & 1 \\
              \end{array}
            \right)>
\]
where $v_\ell(a)=m-1$. So, the image of the mod $\lambda^m$
representation of $f$ contains an $\ell$-group.
\end{corollary}
\begin{proof}
It is well known that if a representation is semi-stable at $p$, the
restriction of $\rho$ on the inertia at $p$ is
\[
\left(
              \begin{array}{cc}
                1 & * \\
                0 & 1 \\
              \end{array}
            \right)
\]
for some $*\ne0$. Since we know that the inertia at $p$ vanishes
modulo $\lambda^{n}$ exactly when $n<m$, then we know that
$*\equiv0\pmod{\lambda^n}$ if and only if $n<m$. Then
$v_\ell(*)=m-1$.
\end{proof}

In \cite{Taixes2008} we computed $10.122$ examples where we can
apply Theorem \ref{teorema}. It is easy to check that most of them
satisfy also the hypothesis of Corollary \ref{corollary}.

In Table \ref{table:examples} we show some of these examples. In
particular, we can see that all but one of them (the one with
$p=13^2$) satisfy also the conditions from Corollary
\ref{corollary}.

We divided the table in $3$ different parts: the first one has some
of the elements with the biggest $\ell$'s that we found. The
greatest one is as big as $1.75\cdot10^{18}$. The next part includes
the elements with a big $p$. Since we worked with elements with
$N\le2000$ and the smallest level appearing is $N=11$, we know that
$p$ can not be bigger than $181$. We have actually precisely one
example with this $p$. Finally, in the last section we have the
couples with the biggest $m$'s. It is remarkable to see that there
is one element with $m=11$.

\begin{table}[!ht]
\[
\begin{array}{c|c|c|c|c|c|c}
N_f               & i & N_g & j & \ell^{m-1}          & p^k    & m \\
\hline
1678 = 2\cdot839  & 8 & 839 & 2 & 1750283935190857471 & 2    & 2\\
1707 = 3\cdot569  & 4 & 569 & 2 & 122272440801294601  & 3    & 2\\
1941 = 3\cdot647  & 4 & 647 & 3 & 5539230441648341    & 3    & 2\\
1839 = 3\cdot613  & 4 & 613 & 3 & 3726338419619653    & 3    & 2\\
1757 = 7\cdot251  & 5 & 251 & 2 & 902088490528867     & 7    & 2\\
1797 = 3\cdot599  & 6 & 599 & 3 & 779881437372101     & 3    & 2\\
1941 = 3\cdot647  & 3 & 647 & 3 & 665741756680589     & 3    & 2\\
1945 = 5\cdot389  & 5 & 389 & 5 & 571255479184807     & 5    & 2\\
1754 = 2\cdot877  & 4 & 877 & 3 & 551522526259063     & 2    & 2\\
1706 = 2\cdot853  & 5 & 853 & 2 & 372293980443053     & 2    & 2\\
1906 = 2\cdot953  & 6 & 953 & 2 & 303408887531093     & 2    & 2\\
1851 = 3\cdot617  & 7 & 617 & 2 & 286866593268389     & 3    & 2\\
\hline
1991 = 11\cdot181 & 4 & 11  & 1 & 27 = 3^3            & 181  & 4\\
1969 = 11\cdot179 & 4 & 11  & 1 & 3                   & 179  & 2\\
1903 = 11\cdot173 & 4 & 11  & 1 & 7                   & 173  & 2\\
1859 = 11\cdot13^2& 8 & 11  & 1 & 3                   & 13^2 & 2\\
1837 = 11\cdot167 & 5 & 11  & 1 & 13                  & 167  & 2\\
\hline
1937 = 13\cdot149 & 4 & 149 & 2 & 59049 = 3^{10}      & 13   & 11\\
1934 = 2\cdot967  & 2 & 967 & 1 & 625 = 5^4           & 2    & 5\\
1929 = 3\cdot643  & 4 & 643 & 2 & 625 = 5^4           & 3    & 5\\
1708 = 2^2\cdot7\cdot61 & 6 & 244 = 2^2\cdot61 & 2 & 81 = 3^4 & 7 & 5\\
1686 = 2\cdot3\cdot281 & 10 & 562 = 2\cdot281 & 4 & 28561 = 13^4 & 3 & 5\\
1643 = 31\cdot53  & 3 & 53  & 2 & 625 = 5^4           & 31   & 5\\
1426 = 2\cdot23\cdot31 & 13 & 713 = 23\cdot31 & 5 & 81 = 3^4 & 2 & 5\\
1401 = 3\cdot467  & 1 & 467 & 2 & 625 = 5^4           & 3    & 5\\
1298 = 2\cdot11\cdot59 & 11 & 649 = 11\cdot59 & 4 & 81 = 3^4 & 2 & 5\\
1158 = 2\cdot3\cdot193 & 13 & 386 = 2\cdot193 & 4 & 625 = 5^4 & 3 & 5\\
1115 = 5\cdot223  & 8 & 223 & 2 & 81 = 3^4            & 5    & 5
\end{array}
\]
\caption{Examples satisfying Theorem \ref{teorema}}
\label{table:examples}
\end{table}

Every pair $(N,i)$ in Table \ref{table:examples} corresponds to the
$i$-th element of the basis of $S_2^{\new}$ sorted with the
\texttt{SortDecomposition} function of Magma \cite{Magma1997}.\\

For any two-dimensional Galois representation $\rho$, let us denote
by $\rho'$ its projectivization. Then we have the following:
\begin{corollary}\label{dihedral}
With the same conditions as in Theorem \ref{teorema}, let $k=1$. Let
us suppose also that $g$ has Complex Multiplication (in this case,
$\im(\rho'_{g,\lambda})$ is a dihedral group). Then the image of
$\rho'_{f,\lambda}$ is not dihedral and the number $m$ of the
Theorem is the smallest one such that the first of the following
inclusions is not an equality:
\[
\mathrm{Dihedral\ group}\subsetneq \overline\rho'_{f,\lambda^m}
\subsetneq PGL_2(\OO_{f,\lambda}/\lambda^m\OO_{f,\lambda}).
\]
\end{corollary}
\begin{proof}
It is clear that for $m-1$,
$\overline\rho'_{g,\lambda^{m-1}}\sim\overline\rho'_{f,\lambda^{m-1}}$,
and since $g$ has CM, $\overline\rho'_{f,\lambda^{m-1}}$ must be a
dihedral group. However, for $m$, since
$\overline\rho'_{f,\lambda^{m}}$ contains an element provided by
Theorem \ref{teorema} which can not be contained in a dihedral
group, it is clear that $\overline\rho'_{f,\lambda^{m}}$ is not a
dihedral group anymore.

For the other inequality it is clear that it is never an equality,
because if it were, $\overline\rho'_{f,\lambda^n}$ would always
equal $PGL_2(\OO_{f,\lambda}/\lambda^{n}\OO_{f,\lambda})$, for every
$n$. And this is impossible, since we know that for $n<m$,
$\overline\rho'_{f,\lambda^n}$ is a dihedral group.
\end{proof}

Let us remark that the conditions in the Theorem are not too
restrictive. For example, just by taking one newform $g$ of level
$N$ with residual mod $\lambda$ representation satisfying the strong
irreducibility condition, minimal with respect to $\lambda$ and not
congruent to any other newform of the same level, using Ribet's
Raising the Level we can find infinitely many examples in which we
can apply our results.

The conditions we are imposing on the pair $(g,\ell)$ are generic in
the following sense: given $g$ they are satisfied for almost every
prime $\ell$. In fact, given $g$ it is well-known that for almost
every prime $\ell$ the representation $\rho_{g,\lambda}$ is
irreducible, as proved by Ribet in \cite{Ribet2} (see also \cite{DV}
for an explicit determination of the finite set of reducible
primes), and as we have already explained the strong irreducibility
condition is automatic if $\ell >3$. It is also well-known that the
number of primes giving congruences between modular forms of fixed
(or bounded) level, called ``congruence primes", is finite: this can
easily be proved by applying Dirichlet's principle (there are only
finitely many cusp forms of bounded level) and the fact that two
newforms that are congruent modulo infinitely many primes must be
equal. Also, the condition of being minimal with respect to
$\lambda$ is equivalent, by Ribet's lowering the level, to the fact
that $g$ is not congruent to some modular form $g'$ of level equal
to a proper divisor of $N$, and so if this condition is not
satisfied $\ell$ has to be a congruence prime and we know that there
are only finitely many of them because the levels of $g$ and $g'$
are both bounded by $N$. We conclude that for any level $N$ there is
constant $C$ such that for any weight $2$ modular form $g$ of level
$N$ and any prime $\ell>C$ the pair $(g,\lambda)$ satisfies the
conditions of the Theorem.

\section{Taylor-Wiles}
To prove Theorem \ref{teorema}, the main result we need is an
extended version of the Taylor-Wiles Theorem. In order to state it,
we have to introduce some notation.

Let $\overline\rho :=\overline\rho_{g,\lambda}$, which we assume to
be strongly irreducible. Let $\Sigma$ be a finite set of prime
numbers. We say that a representation $\rho$ deforming
$\overline\rho$ is of type $\Sigma$ if
\begin{enumerate}
\item $\chi_\ell^{-1}$ det $\rho$ has finite order not divisible
by $\ell$.
\item $\rho$ is minimally ramified outside $\Sigma$.
\item $\rho$ is flat at $\ell$ in the sense of \cite{deShalit} (see also \cite{DaDiTa}).
\end{enumerate}

Let $R_\Sigma$ be the $\OO_{g,\lambda}$-algebra corresponding to the
universal deformation of type $\Sigma$. Let $\Phi_\Sigma$ be the set
of newforms $f$ such that $\rho_{f,\lambda}$ is a deformation of
$\overline\rho$ of type $\Sigma$.

For every $f$ in $\Phi_\Sigma$, consider the map
$R_\Sigma\rightarrow \OO_{f,\lambda}$ corresponding to
$\rho_{f,\lambda}$. We define
$\T_\Sigma\subset\prod_{f\in\Phi_\Sigma}\OO_{f,\lambda}$ as the
image of $R_\Sigma$.

Let $\phi_\Sigma$ be the surjective map
\[
\phi_\Sigma:R_\Sigma\rightarrow\T_\Sigma.
\]

\begin{theorem}[Taylor-Wiles]
Let $\ell$ be an odd prime. If $\ell=3$, let $L=\Q(\sqrt{-3})$ and
suppose $\overline\rho|_{G_L}$ is irreducible. Then $\phi_\Sigma$ is
an isomorphism and $R_\Sigma$ is a complete intersection.
\end{theorem}
\begin{proof}
In \cite{deShalit} and \cite{Diamond1997} this is proved with the
condition $\overline\rho|_{G_L}$ irreducible with
$L=\Q(\sqrt{(-1)^{(\ell-1)/2}\ell})$ and in Proposition \ref{strong
irreducible} we already saw that for $\ell>3$ this condition is
always satisfied.
\end{proof}

\section{Proof of the Theorem}\label{Seccio proof}
We will need first to introduce two auxiliary results.

\begin{proposition}\label{Conditions}
Let $\overline{\rho}$ be a mod $\lambda$ irreducible representation
of conductor $N$, with $\ell>2$. If $\ell=3$, suppose that
$\overline\rho|_{G_L}$ is irreducible. Let us suppose that there
exists only one newform $g$ of weight $2$, trivial character, and
level $N$ such that $\overline{\rho}=\overline{\rho}_{g,\lambda}$.
Let $\QQ$ be the following set of deformation conditions:
\begin{itemize}
\item The deformations are unramified outside $\ell N$.
\item The deformations are minimally ramified everywhere.
\item The determinant of the deformations is the cyclotomic character.
\item The deformations are flat (locally at $\ell$).
\end{itemize}
Then, the deformation ring $\RR_\QQ$ is the ring of integers
$\OO_{g,\lambda}$.
\end{proposition}
\begin{proof}
What we are considering is the problem of deformations of type
$\Sigma = \varnothing$. By the Theorem of Taylor-Wiles, we know that
the universal deformation ring $R_\Sigma$ must be isomorphic to
$\T_{\Sigma}$. By hypothesis, there is only one
$\overline\Q_\ell$-point in $\T_\Sigma$. Then $\RR_\Sigma$ must be
$\OO_{g,\lambda}$ itself.
\end{proof}

\begin{lemma}\label{Only one}
Let $\rho_1$ and $\rho_2$ be two representations, both deforming
$\overline\rho$
\[
\rho_1,\rho_2:\G\rightarrow GL_2(\OO_\lambda/\lambda^n\OO_\lambda)
\]
satisfying the same deformation conditions $\QQ$, such that for
these conditions the universal deformation ring is $\OO_\lambda$.
Then, $\rho_1$ is equivalent to $\rho_2$.
\end{lemma}
\begin{proof}
We suppose they are different. The universal deformation (under
conditions $\QQ$) is
\[
\rho^{univ}:\G\rightarrow GL_2(\OO_\lambda).
\]
Then, we have that there exist two homomorphisms $h_1$ and $h_2$
\[
h_1,h_2:\OO_\lambda\rightarrow \OO_\lambda/\lambda^n\OO_\lambda
\]
such that they induce the identity in the residue fields and also
$h_i\circ\rho^{univ}=\rho_i$. Then $h_1$ and $h_2$ must be different
homomorphisms, but since there exists only one natural projection
from $\OO_\lambda$ to $\OO_\lambda/\lambda^n\OO_\lambda$ fixing the
residue fields, we arrive at a contradiction.
\end{proof}

\begin{proof}[Proof of Theorem \ref{teorema}]
We consider the same set of deformation conditions $\QQ$ as in
Proposition \ref{Conditions}, with $N=N_g$. We consider also the set
of conditions $\QQ'$ as follows:
\begin{itemize}
\item The deformations are unramified outside $\ell pN_g$.
\item The deformations are minimally ramified locally at every place $q \neq p$.
\item The determinant of the deformations is the cyclotomic
character.
\item The deformations are flat locally at $\ell$.
\end{itemize}
So, the set of conditions $\QQ'$ is different from the set of
conditions $\QQ$ only because now we allow ramification at $p$.

By Carayol's result, we know that all such deformations must be in
level $p^kN_g$ with $k\leq2$. Then, by Taylor-Wiles $\RR_{\QQ'}$ is
isomorphic to a Hecke algebra $\T_{\QQ'}$ of level $p^2 N_g$.

Obviously $\overline\rho_{g,\lambda^{m-1}}$ and
$\overline\rho_{g,\lambda^{m}}$ satisfy conditions $\QQ$ and $\QQ'$.
Since
$\overline\rho_{f,\lambda^{m-1}}\sim\overline\rho_{g,\lambda^{m-1}}$,
$\overline\rho_{f,\lambda^{m-1}}$ satisfies also $\QQ$ and $\QQ'$.

By Proposition \ref{Conditions}, $\RR_{\QQ}=\OO_{g,\lambda}$. This
means, by Lemma \ref{Only one}, that if two mod $\lambda^n$
deformations satisfy deformation conditions $\QQ$ they must be the
same. By hypothesis we know that
$\overline\rho_{f,\lambda^{m}}\not\sim\overline\rho_{g,\lambda^{m}}$.
This means that $\overline\rho_{f,\lambda^{m}}$ can not satisfy
conditions $\QQ$. However, $\overline\rho_{f,\lambda^{m}}$ clearly
satisfies conditions $\QQ'$. Since the only difference between both
conditions is the ramification at $p$, the reason for
$\overline\rho_{f,\lambda^{m}}$ not to satisfy $\QQ$ must be
precisely that $\overline\rho_{f,\lambda^{m}}$ ramifies at $p$, as
we wanted to prove.
\end{proof}

\section{Further work}
It would be interesting to improve the main result by relaxing the
assumptions. For example, one should consider in which cases it is
possible to eliminate the condition ``for any other $h\in S_2(N_g),\
\overline\rho_{g,\lambda}\not\sim\overline\rho_{h,\lambda}$" in the
main theorem. In this more general case, the minimal universal
deformation ring will be more complicated, though it is known to be
finite flat complete intersections by the result of Taylor-Wiles.

Looking at Table \ref{table:examples} we saw that $\ell$ and $p$
seem not to be bounded ($p$ is clear). However, we wonder if given
any couple of newforms there is any global bound for $m$.


\begin{thebibliography}{xx}

\bibitem[BCP97]{Magma1997} \textsc{Wieb Bosma, John Cannon and Catherine Playoust},
\textit{The {M}agma algebra system {I}: {T}he user language}.
Journal of Symbolic Computation {\bf24}, 3-4 (1997), 235--265.

\bibitem[Car89]{Carayol1989} \textsc{Henri Carayol},
\textit{Sur les Repr\'{e}sentations Galoisiennes modulo $\ell$
attach\'{e}es aux formes modulaires}. Duke mathematical journal
{\bf59}, 3 (1989), 785--801.

\bibitem[DDT95]{DaDiTa} \textsc{Henri Darmon, Fred Diamond and Richard Taylor},
\textit{Fermat's Last {T}heorem}. International Press (1995).

\bibitem[Dia97]{Diamond1997} \textsc{Fred Diamond},
\textit{An extension of Wiles' results}, in \textit{Modular Forms
and {F}ermat's Last {T}heorem}. Springer, New York, (1997).

\bibitem[dS97]{deShalit} \textsc{Ehud de Shalit}, \textit{Hecke rings and universal deformation rings}, in \textit{Modular Forms
and {F}ermat's Last {T}heorem}. Springer, New York, (1997).

\bibitem[DV00]{DV} \textsc{Luis Dieulefait and Nuria Vila}, \textit{Projective linear groups as Galois groups over $\Q$ via modular representations}. J. Symbolic Comput. {\bf 30} (2000),  799--810.

\bibitem[Rib85]{Ribet2} \textsc{Kenneth A. Ribet}, \textit{On $\ell$-adic representations attached to modular forms II}. Glasgow Math. J. {\bf 27} (1985), 185--194.

\bibitem[Rib90]{Ribet1990} \textsc{Kenneth A. Ribet},
\textit{On Modular Representations of
$\mathrm{Gal}(\overline{{\bf{Q}}}/\bf{Q})$ arising from modular
forms}. Invent. Math. {\bf100} (1990), 431--476.

\bibitem[Rib97]{Ribet} \textsc{Kenneth A. Ribet}, \textit{Images of semistable Galois representations}. Olga Taussky-Todd: in memoriam. Pacific J. Math. (1997), Special Issue,
277--297.

\bibitem[Tai08]{Taixes2008} \textsc{Xavier Taix\'{e}s i Ventosa},
\textit{Theoretical and algorithmic aspects of congruences between
modular {G}alois representations}. Ph.D. Thesis, Institut f\"{u}r
Experimentelle Mathematik (Universit\"{a}t Duisburg-Essen) (2008).

\end{thebibliography}
\end{document}